\documentclass{amsart}
\usepackage[foot]{amsaddr}
\usepackage{amsmath, amsthm, amssymb, amsfonts, graphicx}
\usepackage[all]{xy}
\usepackage{setspace}
\usepackage{hyperref}
\usepackage{tikz-cd}

\onehalfspace

\def\C{\mathbb{C}}

\def\zplus{\mathbb{Z}^+}

\def\th{^\text{th}}

\def\epsilon{\varepsilon}
\def\phi{\varphi}

\def\diag{\text{Diag}}
\def\st{\; | \;}

\DeclareMathOperator{\im}{im}
\DeclareMathOperator{\ind}{Ind}
\DeclareMathOperator{\coker}{coker}
\let\dim\relax
\DeclareMathOperator{\dim}{Dim}

\def\cfm #1{\text{CFM(#1)}}
\def\rcfm #1{\text{B}(#1)}

\newtheorem{thm}{Theorem}[section]
\newtheorem{prop}[thm]{Proposition}
\newtheorem{cor}[thm]{Corollary}
\newtheorem{lemma}[thm]{Lemma}

\theoremstyle{definition}
\newtheorem{ex}[thm]{Example}
\newtheorem{defn}[thm]{Definition}
\newtheorem{rmk}[thm]{Remark}

\theoremstyle{remark}

\begin{document}

\title{A Family of Almost Invertible Infinite Matrices}
\author{Daniel P. Bossaller}

\address{Department of Mathematics and Computer Science, John Carroll University, University Heights, OH, 44118}
\email{dbossaller@jcu.edu}

\maketitle

\begin{abstract}
An algebraic analogue of the family of Fredholm operators is introduced for the family of row and column finite matrices, dubbed ``Fredholm matrices." In addition, a measure is introduced which indicates how far a Fredholm matrix is from an invertible matrix. It is further shown that this measure respects multiplication, is invariant under perturbation by a matrix from $M_\infty(K)$, and is invariant under conjugation by an invertible row and column finite matrix.
\end{abstract}

\section{Introduction}
In the study of rings and algebras, the family of invertible elements is of keen interest, mainly because the presence (or lack) of units is the primary dividing line between rings and fields. As such, many articles have been written studying this family of elements and various generalizations of invertibility. In many cases these articles have opened up new lines of inquiry in ring theory. For example, in \cite{onesidedinverses}, Jacobson studied elements which were one-sided invertible, that is, there exists elements $x$ and $y$ such that $xy = 1$ but $yx \neq 1$. From this work grew the study of directly finite algebras (where every element is left and right invertible) and one of the most well-known open problems in non-commutative ring theory, the Kaplansky Direct Finiteness Conjecture: Every group algebra $KG$ is directly finite, \cite{kapfield}. 

Another example of the primacy of invertibility and its generalizations can be found in the theory of von Neumann regular rings. These are rings where every element $x$ satisfies the following weak invertiblity property: there is an element $y$ such that $xyx = x$. This is broader set of objects than the one-sided invertible elements. In fact, one can see that any one-sided invertible element is von Neumann regular. However, the converse is not true; any idempotent is von Neumann regular but not necessarily one-sided invertible, for example. See \cite{goodearl} for a wide survey on this topic. Notions of invertibility related to von Neumann regularity include strong regularity (also called local invertibility by the author in \cite{BossallerLopez2}), $\pi$-regularity, Drazin inverses, and Moore-Penrose inverses. See \cite{generalizedinverses} for a treatment of the latter two topics and many other invertibility-adjacent topics.

This note introduces a family of nearly invertible infinite matrices which is very closely tied to the one-sided invertibility of Jacobson. These ``algebraically Fredholm" matrices are inspired by their Fredholm operator counterparts in the theory of $C^*$-algebras. These are the bounded linear operators $T$ such that $T + K$ is invertible for some compact operator $K$. For an introduction and explanation of the importance of Fredholm operators functional analysis see \cite{schechter}, many of the results of the present note are inspired by Schechter's treatment of this topic. In the purely algebraic case, i.e. removing the notion of convergence, the analogue of the set of compact linear operators on a Hilbert space $\mathcal H$, $\mathcal K(\mathcal H)$, is the set of infinite matrices indexed by $\zplus \times \zplus$ over the field $\C$ with only finitely many nonzero entries, which will be denoted by $M_\infty(\C)$. It can be shown that $\mathcal K(\mathcal H)$ is the completion, in the appropriate norm of $M_\infty(\C)$. In a similar way, the algebraic analogue of the bounded operators are the row and column finite matrices, $\rcfm \C$. These are the infinite matrices over $\C$ where each row and column has only finitely many nonzero entries. It is straightforward to see that $M_\infty(\C)$ is an ideal of $\rcfm \C$. Furthermore, as shown in \cite{regmul}, $\rcfm \C$ is the largest unital $\C$-algebra which contains $M_\infty(\C)$ as an essential ideal, that is, for any nonzero ideal $J$ has non-trivial intersection with $M_\infty(\C)$. We define the quotient algebra 
\[Q(\C) = \rcfm \C/M_\infty(\C).\]
The results of this article do not require the fact that $\C$ is algebraically closed. Due to this, we assume that the entries of our matrices come from some field $K$ of characteristic zero.

\section{Algebraic Fredholm Theory}
In many articles concerning the classification of extensions of $C^*$-algebras, for instance \cite{paschkesalinas}, the Fredholm index of an operator is used to parameterize the various extensions of an algebra $A$ by $B$. The goal of this article is to introduce an algebraic version of this index by examining matrices $A \in \rcfm K$ whose image under the canonical surjection $\pi: \rcfm K \mapsto \rcfm K / M_\infty(K)$ is invertible.

First, let us establish the following, easily verified facts about $M_\infty(K)$. Suppose that $A \in \rcfm K$; we adopt the notation that $\mathcal L_A$ and $\mathcal R_A$ denote the linear transformations (on a countably infinite dimensional, $K$-vector space $V$), $\mathcal L_A(x) = Ax$ and $\mathcal R_A(y) = yA$ for a column vector $x$ and row vector $y$.

\begin{lemma}
Let $A$ be any infinite matrix.
\begin{enumerate}
\item $A \in M_\infty(K)$ if, and only if, $\mathcal L_A$ and $\mathcal R_A$ are linear transformations of $V$ and  $\dim(\im \mathcal L_A) < \infty$ and $\dim(\im \mathcal R_A) < \infty$. 
\item$M_\infty(K)$ is the unique minimal ideal of $\rcfm K$.
\end{enumerate}
\end{lemma}
\begin{proof}
The first is established through straightforward calculation, then it follows that $M_\infty(K)$ is an ideal of $\rcfm K$. To show minimality of $M_\infty(K)$, suppose there were an ideal $\{0\} \neq I \subsetneq M_\infty(K)$. Then there must be some nonzero element $a \in I$ such that $\rcfm K a \rcfm K \neq \{0\}$ for otherwise, $\mathcal L_a(x) = 0$ and $\mathcal R_a(x) = 0$ for all $x \in \rcfm K$, which implies that $a = 0$ since $I_\infty(K) \in \rcfm K$. Write $a$ in terms of the matrix units of $M_\infty(K)$: \[a = k_{ij} e_{ij} + \sum_{mn} k_{mn} e_{mn}\] such that $m \neq i$ or $n \neq j$. Then any matrix unit $e_{kl}$ may be written as $e_{kl} = e_{ki}ie_{jl}$. Thus the two-sided ideal generated by $a$ must equal $M_\infty(K)$. Now suppose that there were another minimal ideal $J$ of $\rcfm K$. Let $a \in J$. By construction, there is at least one $e_{ii} \in M_\infty(K)$ such that $e_{ii}a \neq 0$, thus $e_{ii}a \in M_\infty(K) \cap J$, which is an ideal of $\rcfm K$ contained in $M_\infty(K)$. This is a contradiction of the minimality of $M_\infty(K)$ unless $J = M_\infty(K)$.
\end{proof}

\begin{defn}
A matrix $A \in \rcfm K$ is called {\bf algebraically Fredholm} if the image of $A$ under the natural surjection $\pi: \rcfm K \rightarrow Q(K)$ is invertible in $Q(K)$. In the remainder of this article, we will refer to such matrices simply as ``Fredholm matrices."
\end{defn}

In other words, $A$ is Fredholm if and only if  there exists $A_1, A_2 \in \rcfm K$ and $S_1, S_2 \in M_\infty$ such that $AA_1 = I_\infty - S_1$ and $A_2 A = I_\infty - S_2$. A consequence of this characterization: since $A_1$ and $A_2$ differ only by some element $R \in M_\infty(K)$, we can select some $A_0$ as which functions as both a left and a right Fredholm inverse. That is, there exists $A_0$, $S_1$ and $S_2$ such that $A A_0 = I_\infty + S_1$ and $A_0 A = I_\infty + S_2$. Note that this also implies that Fredholm inverses are unique up to perturbation by some element of $M_\infty(K)$; when we say ``the" Fredholm inverse of a matrix, it is understood within this context.

\begin{prop}
The family of Fredholm matrices is closed under multiplication.
\end{prop}
\begin{proof}
Say that $A$ and $B$ are matrices such that $\overline A$ and $\overline B$ are invertible in $Q(K)$. Suppose $A$ has Fredholm inverse $A_0$ and $B$ has Fredholm inverse $B_0$. Then consider $\overline{AB} = \bar{A} \bar{B}$ which is clearly invertible in $Q(K)$ with Fredholm inverse $\bar{B_0}\bar{A_0}$.
\end{proof}

In the theory of Banach algebras, Fredholm operators are defined as those operators $T$ with closed range such that $\dim(\ker(T))$ and $\dim(\ker(T'))$ are finite (where $T'$ is the Hilbert space adjoint of $T$). The following result establishes that if a matrix is Fredholm, then $\dim(\ker(T))$ and $\dim(V/TV)$ are finite. Furthermore, using this fact we will introduce the ``index" of a matrix which will function as a measurement for how far a given Fredholm matrix is from being invertible. In this article, we will use the notation $\im(A)$, $\ker(A)$, and $\coker(A)$ for the image, kernel, and cokernel of the linear transformation $\mathcal L_{A}$.

\begin{lemma}\label{fredholm dimension}
Let $A \in \rcfm K$ be a Fredholm matrix. Then $\ker(A)$ and $\coker(A)$ are finite dimensional subspaces of $V$.
\end{lemma}
\begin{proof}If $A$ is Fredholm, then there exists $A_0\in \rcfm K$ and $R, S \in M_\infty(K)$ be matrices such that $A_0A = I_\infty - R$ and $AA_0 = I_\infty - S$.

To show that $\ker(A)$ is finite dimensional, suppose that there is an infinite, linearly independent set of elements in $\ker(A)$, $\{b_i \st i \in \zplus\}$. Then $Ab_i = 0$ for each $i \in \zplus$. Construct a matrix $B = (b_1 \; |\; b_2\; |\; \cdots\; )$. By construction $AB = 0$.  Then
\[0 = A_0(AB) = (A_0A)B = B + RB.\] Thus $B = -RB$; this is a contradiction since $RB$ is a matrix with finitely many nonzero rows, but $B$, being constructed from an infinite linearly independent set of vectors, must have infinitely many nonzero rows. Thus for a Fredholm matrix $A$, $\ker(A)$ must be finite dimensional. 

For the other claim, first note that $\im(AA') \subseteq \im(A)$; thus $\im(I_\infty - S) \subseteq \im(A)$ which means that $\coker(A) \subseteq \coker(I_\infty - S)$. Thus 
\[\dim(\coker(A)) \leq \dim(\coker(I_\infty - S)).\]
We claim that $\dim(\coker(I_\infty - S))$ is finite-dimensional. Note that
\[\coker(I_\infty - S) = V / \im(I_\infty - S) = \{v \in V \st v = S v\} \subseteq \im(S).\] Since the dimension of $\im(S)$ is finite, the dimension of the cokernel must be finite also, which proves the claim.
\end{proof}

\begin{rmk}
In the study of Fredholm operators in functional analysis, the corresponding result to the previous is a bijection. However, it is an open question as to whether this result can be extended similarly. The key problem is the existence of row and column finite matrices which are not invertible in $\rcfm K$ but are invertible in $\cfm K$. An example of one such pair of matrices is the following:
\[P = \begin{pmatrix}
1 &-1 &0 &\cdots\\
0 &1 &-1 &\cdots\\
0 &0 &1 &\cdots\\
\vdots &\vdots &\vdots &\ddots
\end{pmatrix} \text{ and } P^{-1} = \begin{pmatrix}
1 &1 &1 &\cdots\\
0 &1 &1 &\cdots\\
0 &0 &1 &\cdots\\
\vdots &\vdots &\vdots &\ddots
\end{pmatrix}.\]
In the theory of Banach spaces, one may appeal to the Bounded Inverse Theorem which guarantees that a bounded operator $T: X \rightarrow Y$ which has $\im(T) = Y$ and $\ker(T) = \{0\}$ has an inverse $T^{-1}$ which is also bounded. The proof of the analytic analogue involves restricting the Fredholm operator $T$ to a operator which is guaranteed to have a bounded inverse.
\end{rmk}

\begin{ex}
Many of the intuitions developed in a linear algebra course break down in the face of infinite matrices. The dimensions of the kernel and cokernel of a linear transformation is precisely one of those cases. Recall the rank-nullity theorem; given an endomorphism $f$ of an $n$-dimensional $K$-vector space $V$, $\dim(\im(f)) + \dim(\ker(f)) = n$. The dimension of the cokernel $V/f(V)$ equals the dimension of the kernel since \[\dim(\coker(f)) = n - \dim(\im(f)) = \dim(\ker(f)).\]

To see how this intuition fails in the case of infinite matrices, consider the matrix $S_{-1} = \sum_{i = 1}^\infty e_{i,i+1}$ where $e_{ij}$ is the matrix unit with $1$ in the $(i,j)\th$ entry and zeroes elsewhere. The matrix $S_{-1}$ can be visualized as the infinite matrix which has $1$'s along the first super-diagonal. It's clear that $S_{-1}$ is Fredholm with inverse $S_1 = \sum_{j = 1}^\infty e_{j+1,j}$. The kernel of $\mathcal L_{S_{-1}}$ is the following subspace of $V$, $\{(k, 0, 0, \ldots)^T \st k \in K\}$, which has dimension 1. However the cokernel $V/\im(S_1) = \{0\}$, which has dimension 0. We'll return to this matrix and other similar matrices after the following definition.
\end{ex}

Since we can think of Fredholm matrices as those which are ``almost invertible," we propose the following as a measurement of how close a given Fredholm matrix is to being invertible

\begin{defn}
Let $A$ be a Fredholm matrix, we define the {\bf index} of $A$ to be 
\[\ind(A) = \dim(\ker(A)) - \dim(\coker(A)).\]
\end{defn}

\begin{ex}\label{Si generators}
Define the following set of matrices. For an integer $i$, $S_i$ will denote the matrix \[S_i = \sum_{j = 1}^\infty e_{i+j,j} \text{ for all } j \text{ such that } j>-i.\] These are the matrices which have $1$'s along the $i\th$ sub-diagonal and zeroes elsewhere; in the case where $i < 0$, then the matrices have $1$'s along the $-i\th$ super-diagonal. Because these matrices amount to shift operators on the infinite-dimensional vector space $V$, it is straightforward to calculate kernels and cokernels of the matrices and see that 
\[\ind(S_i) = -i\]
\end{ex}
We now establish the following facts about the index.
\begin{prop}\label{fredholm properties}
Let $A$ and $B$ be Fredholm matrices and let $T \in M_\infty(K)$, and let $A_0$, $R$, and $S$ be as in the proof of Lemma \ref{fredholm dimension}.
\begin{enumerate}
    \item $\ind(AB) = \ind(A) + \ind(B)$.
	\item $\ind(A_0) = - \ind(A)$. 
    \item $A+T$ is Fredholm, and $\ind(A + T) = \ind(A)$.
\end{enumerate}
\end{prop}
\begin{proof}
To prove (1), we divide $V$ up into four subspaces $V_1$, $V_2$, $V_3$, and $V_4$.
\[\begin{array}{r c l}
V_1 &= &\ker(A) \cap \im(B),\\
\im(B) &= &V_1 \oplus V_2,\\
\ker(A) &= &V_1 \oplus V_3 \text{, and from here we get}\\
V &= &\im(B) \oplus V_3 \oplus V_4.
\end{array}\]

Note that by Lemma \ref{fredholm dimension}, $V_1$, $V_3$, and $V_4$ are finite dimensional subspaces of $V$ since they are subspaces of $\ker(A)$ or $\coker(B)$ which are finite dimensional. Let $d_i = \dim(V_i)$ for $i \in \{1, 3, 4\}$. In additon we can find two more subspaces $W, X \subseteq V$ by writing $\ker(AB) = \ker(B) \oplus W$ and $\im(A) = \im(AB) \oplus X$. Since $W \subseteq \ker(AB)$ and $X \subseteq \coker(AB)$, both $W$ and $X$ are finite dimensional. Note that $W$ is the subspace of all vectors $v$ such that $v \in \im(B)$ but $v \in \ker(A)$, so $\dim(W) = d_1$. Also note that $\im(A) = \mathcal{L}_A(V) = \mathcal{L}_A(\im(B) \oplus V_3 \oplus V_4) = \im(AB) \oplus \mathcal{L}_A(V_4)$. Since $\ker(A) = V_1 \oplus V_3$, $\mathcal{L}_A$ must be a one-to-one linear transformation from $V_4$ to $W$ which implies that $V_4$ and $X$ must have the same dimension, $\dim(X) = d_4$.

Collecting our work from the previous paragraphs, we have that 
\[\begin{array}{r c l}
\dim(\ker(AB) &= &\dim(\ker(B)) + d_1\\
\dim(\coker(AB)) &= &\dim(\coker(A)) + d_4\\
\dim(\ker(A)) &= &d_1 + d_3\\
\dim(\coker(B) &= &d_3 + d_4
\end{array}.\]
So we calculate $\ind(AB) = \dim(\ker(B)) + d_1 - \dim(\coker(A)) - d_4$. On the other hand, $\ind(A) + \ind(B) = \dim(\ker(A)) - \dim(\coker(A)) + \dim(\ker(B)) - \dim(\coker(B)) = d_1 + d_3 - \dim(\coker(A)) + \dim(\ker(B)) - d_3 + d_4$, which gives the desired equality.

The proof of (2) follows from the fact that $\ker(I_\infty - R) = \{v \in V \st v - R v = 0\} = \{v \in V \st v = R v\} = \coker(I_\infty - R)$. Because those subspaces have finite dimension, we calculate 
\[0 = \ind(I_\infty - S) = \ind(A_0 A) = \ind(A_0) + \ind(A).\]

To show that (3) holds, define $R' = (R -A_0T)$ and $S' = (S - T A_0)$, and note \[\begin{array}{r c l} A_0(A + T) &= &I_\infty - R + A_0 T = I_\infty - R'\\
(A+T)A_0 &= &I_\infty - S + TA_0 = I_\infty - S'
\end{array}.\] Thus $A+T$ is Fredholm (with the same inverse to boot). Finally,
\[\ind(A_0) + \ind(A+T) = \ind(A_0(A+T) = \ind(I_\infty - R') = 0.\] Since $\ind(A_0) = - \ind(A)$, we have that $\ind(A) = \ind(A+T)$. 
\end{proof}

We finish this section with the following corollary.
\begin{cor}\label{conjugation}
Let $U$ be an invertible matrix in $\rcfm K$, and let $A$ and $B$ be Fredholm matrices. Furthermore, let $B_0$ be a Fredholm inverse of $B$. Then the following properties hold.
\begin{enumerate}
\item $\ind(U) = 0$
\item $\ind(A) = \ind(U^{-1} A U)$.
\item $\ind(A) = \ind(B_0 A B)$.
\end{enumerate}
\end{cor}
\begin{proof}
The first claim follows from the fact that if $U$ is invertible, its kernel and cokernel are trivial. The second follows from the first, and the third follows from Proposition \ref{fredholm properties}.
\end{proof}

\section{Embeddings of the Toeplitz-Jacobson Algebra into B(K)}

The {\bf Toeplitz-Jacobson algebra} is the $K$-algebra with presentation
\[\mathcal T = \langle x,y \st xy = 1\rangle\] which was first investigated by Jacobson in \cite{onesidedinverses}. This algebra can be thought of as being generated by an element $x$ which is right invertible, but not left invertible and its right inverse $y$. Jacobson's article also included the first of the following embedding of $\mathcal T$ into the $K$-algebra of row and column finite matrices $\rcfm K$.

\begin{ex} (\cite{onesidedinverses}, Equation 6)
Recall the definition of $S_i$ from Example \ref{Si generators}:
\[S_i = \sum_{j = 1}^\infty e_{i+j,j} \text{ for all } j \text{ such that } j>-i.\] One can see that there is an isomorphism $\Phi: \mathcal T \rightarrow \langle S_{-1}, S_1 \rangle \subseteq \rcfm K$ such that  $\Phi(x) = S_{-1}$ and $\Phi(y) = S_1$. For ease of reference, we will call this embedding the ``Jacobson embedding" of $\mathcal T$. 

However, the Jacobson embedding of $\mathcal T$ is far from the only embedding of $\mathcal T$ into $\rcfm K$. Two more are given by $\Psi, \Xi: \mathcal T \rightarrow \rcfm K$, where 
\[\Psi(x) = T_{-1} = \sum_{i=1}^\infty \frac{1}{i+1}e_{i,i+1} \text{ and } \Psi(y) = T_{1} = \sum_{j =1}^\infty (j+1) e_{j+1,j}\]
and 
\[\Xi(x) = S_{-2} \text{ and } \Xi(y) = S_2.\]
\end{ex}

From Theorem 4 of \cite{onesidedinverses} and the construction of the three homorphisms in the previous example, we have the following fact:
\begin{prop}
Each of the three homomorphisms $\Phi$, $\Psi$, and $\Xi$ from the previous example are injective maps embedding $\mathcal T$ into $\rcfm K$.
\end{prop}

We close this article with an example which shows how the index may be used to classify and distinguish between these three embeddings.

\begin{ex}
In \cite{Bossaller1}, the author defines two embeddings $E_1$ and $E_2$ of $\mathcal T$ into $\rcfm K$ as {\bf equivalent} if there is an invertible $U \in \rcfm K$ which conjugates $E_1$ into $E_2$.

With this notion of equivalence in mind, one can see that the embeddings determined $\Phi$ and  $\Psi$ share an equivalence class. The matrix \[U = \diag(1, 2, 6, \ldots, n!, \ldots)\] is an invertible row and column finite matrix such that $US_i U^{-1} = T_i$ for $i \in \{1,2\}$. Because the $S_i$ and $T_j$ are generators for $\Phi(\mathcal T)$ and $\Psi( \mathcal T)$, respectively, $U$ must conjugate the first into the second, thus the embeddings share the same equivalence class.

Note however that $\Phi(\mathcal T)$ and $\Xi(\mathcal T)$ cannot share the same equivalence class. For if there were an invertible row and column finite matrix $V$ which conjugates $S_1$ into $S_2$ and $S_{-1}$ into $S_{-2}$, Corollary \ref{conjugation} would then force $S_1$ and $S_2$ to have the same index, which is not true by Example \ref{Si generators}.
\end{ex}
\bibliographystyle{plain}
\bibliography{fredholm}
\end{document}